\documentclass{article}
\usepackage[utf8]{inputenc}
\usepackage{amsmath}
\usepackage{amssymb}
\usepackage{amsthm}
\usepackage{verbatim}
\usepackage{enumerate}
\usepackage[noadjust]{cite}

\newtheorem{theorem}{Theorem}
\newtheorem{conj}{Conjecture}

\title{On L(2,1)-labelings of oriented graphs}

\author{
  Lucas Colucci$^{1,2}$\\
  \texttt{lucascolucci@renyi.hu}
  \and
  Ervin Gy\H ori$^{1,2}$\\
  \texttt{gyori.ervin@renyi.mta.hu}
}

\date{  $^1$Alfréd Rényi Institute of Mathematics, Hungarian Academy of Sciences, Re\'altanoda~u.~13--15, 1053 Budapest, Hungary\\%
        $^2$Central European University, Department of Mathematics and its Applications, N\'ador~u.~9, 1051 Budapest, Hungary\\[2ex]%
        \today}

\begin{document}
\maketitle

\begin{abstract}
   We extend a result of Griggs and Yeh about the maximum possible value of the $L(2,1)$-labeling number of a graph in terms of its maximum degree to oriented graphs. We consider the problem both in the usual definition of the oriented $L(2,1)$-labeling number and in some variants we introduce.
\end{abstract}

\section{Introduction}

A \emph{$L(2,1)$-labeling}, or \emph{$L(2,1)$-coloring}, of a graph $G$ is a function $f:V(G)\rightarrow\{0,\dots,k\}$ such that $|f(u)-f(v)|\geq 2$, if $uv \in E(G)$; and $|f(u)-f(v)|\geq 1$, if there is a path of length two joining $u$ and $v$. The minimum value of $k$ among the $L(2,1)$-labelings of $G$ is denoted by $\lambda_{2,1}(G)$, and it is called the \emph{$L(2,1)$-labeling number} of $G$. This notion was introduced by Yeh \cite{yeh1990labeling}, and it traces back to the frequency assignment problem of wireless networks introduced by Hale \cite{hale1980frequency}. 

\medskip

The definitions above can be extended to oriented graphs (a directed graph whose underlying graph is simple), namely: if $G$ is an oriented graph, a $L(2,1)$-labeling of $G$ is a function $f:V(G)\rightarrow\{0,\dots,k\}$ such that $|f(u)-f(v)|\geq 2$, if $uv \in E(G)$; and $|f(u)-f(v)|\geq 1$, if there is a \emph{directed }path of length two joining $u$ and $v$. The corresponding $L(2,1)$-labeling number is usually denoted by $\overrightarrow{\lambda}_{2,1}(G)$. These labelings were first considered by Chang and Liaw \cite{chang20032}, and the $L(2,1)$-labeling problem has been extensively studied since then in both undirected and directed versions. We refer the interested reader to the excellent surveys of Calamoneri \cite{calamoneri2011h} and Yeh \cite{yeh2006survey}.

\medskip

One of the most basic results about $L(2,1)$-labelings, which appeared in the seminal paper of Griggs and Yeh \cite{griggs1992labelling}, is an asymptotically sharp upper bound on $\lambda_{2,1}(G)$ as a function of $\Delta$, the maximum degree of the graph. On the one hand, they proved that there is a greedy $L(2,1)$-labeling of $G$ with $k \leq \Delta^2+2\Delta$; on the other hand, every $L(2,1)$-labeling of the incidence graph of a projective plane requires $k \geq \Delta^2-\Delta$. They conjectured that the stronger bound $\lambda_{2,1}(G) \leq \Delta^2$ holds for every $G$, which was proved by Havet et al. \cite{havet20082} for sufficiently large values of $\Delta$.

\medskip

In this note, we will address the problem of bounding the $L(2,1)$-labeling number asymptotically in directed graphs. Our results are divided into two sections: in Section \ref{classic}, we will consider the asymptotic value of the $L(2,1)$-labeling number of oriented graphs as it is defined above. In Section \ref{new}, we introduce alternative definitions of this number and deal with the corresponding problems in these new settings.

\section{Classical directed graph version}\label{classic}

\medskip

Even though there is a bound on $\overrightarrow{\lambda}_{2,1}(G)$ in terms of $\lambda_{2,1}(H)$, where $G$ is an oriented graph and $H$ is its underlying graph, namely, $\overrightarrow{\lambda}_{2,1}(G)\leq\lambda_{2,1}(H)$, it is usually far from sharp. Indeed, these two quantities behave quite differently: while it is easy to see that $\Delta(H)+1\leq \lambda_{2,1}(H)$ (as every vertex in a neighbor of a vertex in $H$ must be labeled with a different number), there is no such phenomenon in the oriented case, in which the neighborhood of any vertex can be locally colored with two colors, one for the in-neighborhood and other for the out-neighborhood. In fact, there is no lower bound on $\overrightarrow{\lambda}_{2,1}(G)$ in terms of its maximum degree: for instance, every directed tree $T$ satisfies $\overrightarrow{\lambda}_{2,1}(T)\leq 4$ (\cite{chang20032}). On the other hand, for an undirected tree $T$, $\Delta(T)+1\leq \lambda_{2,1}(T)\leq \Delta(T)+2$ (\cite{griggs1992labelling}). Similar contrasting results hold for broader classes of oriented planar graphs (see, e.g., \cite{calamoneri20132}).

\medskip

Motivated by these differences, we show in the following theorem that, for oriented graphs, we can give a sharper bound on $\overrightarrow{\lambda}_{2,1}(G)$ as a function of the maximum degree inside a block (i.e., a maximal biconnected subgraph) of the underlying graph of $G$ (in contrast to its global maximum degree). We also show a construction that yields a lower bound asymptotically equal to half of the upper bound.

\medskip

\begin{theorem}
Let $G$ be an oriented graph with the following property: for every block $B$ of its underlying graph, all the in- and outdegrees of the vertices of $G[B]$ are bounded by $k$. Then $\overrightarrow{\lambda}_{2,1}(G)\leq 2k^2+6k$.
\end{theorem}

\begin{proof}

We proceed by induction on the number of blocks of $H$, the underlying graph of $G$. If $H$ has only one block (that is, it is 2-connected), it is clear that we can color $G$ greedily using at most $2k^2+6k+1$ colors, since the first (resp. second) directed neighborhood of any vertex $v$ in $G$ contains at most $2k$ (resp. $2k^2$) vertices, and each of those vertices forbids at most three (resp. one) colors for $v$.

\medskip

On the other hand, if $H$ contains at least two blocks, let $v$ be a cut vertex with the property that at most one of the blocks containing $v$ contains a cut vertex distinct from $v$. It is clear that such a vertex exists from the tree structure of the blocks of $H$. Let $B_1,\dots,B_t$ be the blocks containing $v$ such that $v$ is the only cut vertex of $B_i$.

\medskip

We apply induction on the graph $G'=G-\bigcup_{i=1}^t(V(B_i)\backslash\{v\})$ to get a coloring of it using at most $2k^2+6k+1$ colors. We are left with the vertices of the blocks $B_i$ (except $v$) to color.

\medskip

Let $A$ and $B$ be, respectively, the set of uncolored vertices that point to and from $v$ in $G$. It is clear that the size of any connected component in $A$ and $B$ is at most $k$ and that the only paths joining these components pass through $v$. In this way, as $v$ has at most $2k$ colored neighbors in $G$ at this point, we have at least $2k^2+6k+1-2k-3 \geq 2k$ distinct free colors for the vertices in $A$ and $B$. Let some of the free colors be $c_1<c_2<\dots<c_{2k}$. We use colors $c_1,c_3,\dots,c_{2k-1}$ for $A$ and $c_2,c_4,\dots,c_{2k}$ for $B$, coloring each vertex in a connected component with a distinct color.

\medskip

Now that $A\cup B$ is colored, we have to color the vertices of $\bigcup_{i=1}^tB_i$ at distance at least two from $v$. We can color these vertices greedily as before, since its neighbors and second neighbors lie inside a block of $H$, in which the maximum degree is $k$.
\end{proof}

The construction, as we show in the next theorem, is more sophisticated than the corresponding one for the undirected case:

\begin{theorem}\label{const}
There is an oriented graph $G$ such that its underlying graph is 2-connected, every indegree and outdegree in $G$ is bounded by $k+O(1)$ and $\overrightarrow{\lambda}_{2,1}(G)\geq k^2 + O(k)$.
\end{theorem}

\begin{proof}
Let $V(G)=\mathbb{Z}_k^2$, where $k \geq 4$ is a positive integer. To simplify the notation, we write $ab$ for the pair $(a,b)\in\mathbb{Z}_k^2$. The arcs of $G$ are defined as follows, where the operations are considered modulo $k$:

\begin{enumerate}[i.]
    \item $ab \rightarrow bc$, if $c>a$.
    
    \item $ab \rightarrow (b+1)c$, if $c\leq a$ and $c \ne a-1$.
    
    \item $ab \rightarrow a(b+1)$, if $a\ne b+2$.
    
    \item $ab \rightarrow (a+1)b$, if $a \ne b+1$.
    
\medskip

It is easy to check that $G$ does not contain opposite arcs and both the indegree and outdegree of its vertices are bounded by $k+1$. Furthermore, it will be clear from the proof that its underlying graph is 2-connected.

\medskip 

Note that to prove that the theorem it suffices to show that, for every pair of vertices $ab$, $cd$ with $a,b,c,d \notin \{0,k-1\}$, there is a directed path of length at most 2 from $ab$ to $cd$ or vice-versa. Therefore, we assume this condition holds in what follows.

\medskip

We can find paths of length at most 2 joining $ab$ and $cd$ as follows:

\begin{enumerate}[1.]
    \item If $a<c$ and $b<d$: $ab \rightarrow bc \rightarrow cd$.
    
    \item If $a>c$ and $b>d$: $cd \rightarrow da \rightarrow ab$.
    
    \item If $a<c$ and $b>d$: $cd \rightarrow (d+1)a \rightarrow ab$, except if:
        \begin{enumerate}[i.]
            \item $c=a+1$: $ab \rightarrow (b+1)a \rightarrow (a+1)d$. 
            
            \item $b=d+1$: $cd \rightarrow (d+1)(a-1) \rightarrow a(d+1)$.
        \end{enumerate}
    
    \item If $a>c$ and $b<d$: $ab \rightarrow (b+1)c \rightarrow cd$, except if:
        \begin{enumerate}[i.]
            \item $a=c+1$: $(c+1)b \rightarrow (b+1)(c-1) \rightarrow cd$. 
            
            \item $d=b+1$: $ab \rightarrow (b+1)(c-1) \rightarrow c(b+1)$.
         \end{enumerate}   
         
    \item If $a=c$ and, say, $b<d$ (without loss of generality): $ab \rightarrow (b+1)a \rightarrow ad$, except if:
        \begin{enumerate}[i.]
            \item $d=b+1$ and $a \neq b+2$: $ab \rightarrow a(b+1)$.
            
            \item $d=b+1$ and $a=b+2$: $a(b+1) \rightarrow ab$.
         
         \end{enumerate}   
    \item If, say, $a<c$ (without loss of generality) and $b=d$: $ab \rightarrow b(c-1) \rightarrow cb$, except if:
    
        \begin{enumerate}[i.]
            \item $c=a+1$ and $a \ne b+1$: $ab \rightarrow (a+1)b$.
            
            \item $c=a+1$ and $a=b+1$: $(a+1)b \rightarrow ab$.
        \end{enumerate}

        \end{enumerate}

\end{enumerate}

\end{proof}

\section{Other directed versions}\label{new}

Many different generalizations of the $L(2,1)$-labeling problem have been investigated. The $L(h,k)$-labeling is probably the most famous of them: it is defined as a coloring of the vertices of a graph (either undirected or directed) with integers $\{0,\dots,n\}$ for which adjacent vertices get colors at least $h$ apart, and vertices connected by a path of length 2 get colors at least $k$ apart. When the interval is considered as a cycle (and hence, for instance, the colors 0 and $n$ are just 1 apart), we get yet another new variant. Again, we refer to the survey of Calamoneri \cite{calamoneri2011h} as a comprehensive list of results and references about those and other related problems.

\medskip

In this section, we propose other versions of the problem. A path of lenght two admits three pairwise non-isomorphic orientations: $a \rightarrow b \rightarrow c$, $a \rightarrow b \leftarrow c$, and $a \leftarrow b \rightarrow c$; we call these paths $P_1$, $P_2$ and $P_3$, respectively. In this terminology, we can rephrase the definition of a $L(2,1)$-labeling of an oriented graph $G$ as follows: an assignment $f:V(G)\rightarrow\{0,\dots,k\}$ such that $|f(u)-f(v)|\geq 2$, if $uv \in E(G)$; and $|f(u)-f(v)|\geq 1$, if there is a $P_1$ in $G$ joining $u$ and $v$.

\medskip

We study the corresponding problems that arise when we replace $P_1$ in this definition by $P_2$ or $P_3$, or, even more generally, by a subset $S$ of $\{P_1,P_2,P_3\}$. We denote the corresponding minimum value of $k$ by $\lambda_{S}(G)$. Some of the choices of $S$ lead us back to previous questions, namely, $\lambda_\emptyset(G)=2\chi(G)-1$; $\lambda_{\{P_1,P_2,P_3\}}(G)=\lambda_{2,1}(H)$, where $H$ is the underlying graph of $G$; and $\lambda_{\{P_1\}}(G)=\overrightarrow{\lambda}_{2,1}(G)$. Also, by the symmetry of $P_2$ and $P_3$, we have just the following three cases left to consider: $S=\{P_2\}$, $S=\{P_2,P_3\}$ and $S=\{P_1,P_2\}$.

\medskip

In each one of those cases, we are going to determine the order of magnitude, and, with one exception, the correct asymptotic value, of the maximum possible value of $\lambda_S(G)$ in terms of the maximum degree of $G$.

\medskip

First, we consider $S=\{P_2\}$, i.e., when the only two path considered is $a \rightarrow b \leftarrow c$. We have the following asymptotically sharp result:

\begin{theorem}\label{thm2}
Let $G$ be an oriented graph such that $d_+(v) \leq k$ and $d_-(v)\leq k$ for all $v \in V(G)$. Then $\lambda_{\{P_2\}}(G) \leq k^2+O(k)$, and there is a family of graphs that matches this upper bound asymptotically.
\end{theorem}

\begin{proof}
We color $G$ greedily with the colors $\{0,\dots,k^2+5k\}$: given a vertex $v$, each of its at most $2k$ neighbors forbid at most 3 colors for $v$. Among the second neighbors, only the at most $k(k-1)=k^2-k$ vertices that are joined by a $P_2$ to $v$ forbid colors for $v$, at most one new color per vertex. In total, at most $3 \cdot 2k+k^2-k=k^2+5k$ colors are forbidden for $v$. 

\medskip

As for the sharpness of the bound, the same construction as in the undirected case works. Let $G=(A,B,E)$ be the oriented bipartite incidence graph of a projective plane with point set $A$, line set $B$, $|A|=|B|=k$, and all the edges pointing from $A$ to $B$. Both the in- and outdegrees of $G$ are bounded by $(1+o(1))\sqrt{k}$ and there is a $P_2$ joining every pair of vertices in $A$. Therefore, at least $n$ different colors are needed in any valid labeling of $G$.
\end{proof}

\medskip

In the case $S=\{P_2,P_3\}$, we have the following result, which does not yield an asymptotic sharp bound, but a factor 2 for the ratio between the upper and lower estimates:

\medskip

\begin{theorem}\label{thm23}
Let $G$ be an oriented graph such that $d_+(v) \leq k$ and $d_-(v)\leq k$ for all $v \in V(G)$. Then $\lambda_{\{P_2,P_3\}}(G) \leq 2k^2+O(k)$. On the other hand, there is a family of graphs $G$ with $d_+(v)= (1+o(1))k$, $d_-(v)= (1+o(1))k$ for every $v \in V(G)$ and $\lambda_{\{P_2,P_3\}}(G)\geq k^2+O(k)$.
\end{theorem}

\medskip

The proof of the upper bound in Theorem \ref{thm23} is obtained in a similar way as in Theorem \ref{thm2}, i.e., coloring the graph greedly, bounding the number of forbidden colors for a given vertex using the sizes of its first and second neighboorhoods. The lower bound comes from the very same construction in Theorem \ref{thm2}. We omit the details. 

\medskip

Finally, in the case $S=\{P_1,P_2\}$, we have a different upper bound and an asymptotically sharp construction, as stated in the following theorem:

\begin{theorem}
Let $G$ be an oriented graph such that $d_+(v) \leq k$ and $d_-(v)\leq k$ for all $v \in V(G)$. Then $\lambda_{\{P_1,P_2\}}(G) \leq 3k^2+O(k)$. Furthermore, there is a family of graphs that matches this bound asymptotically.
\end{theorem}

\begin{proof}
Again, we apply the greedy algorithm as in Theorem \ref{thm2} to get the upper bound. 

\medskip 
On the other hand, consider the following construction: if $H=(A,B,E)$ is the bipartite incidence graph of a projective plane with point set $A$, line set $B$ and $|A|=|B|=k$ with all edges oriented from $A$ to $B$, let $H'=(A',B',E')$ and $H''=(A'',B'',E'')$ be two copies of $H$ with edges oriented from $A'$ to $B'$ and $A''$ to $B''$, respectively. For a vertex $p \in V(H)$, we denote by $p'$ (resp. $p''$) its copy in $H'$ (resp. $H''$), and we call $p, p',p''$ twin vertices. We construct an oriented graph $G$ as follows: The vertex set of $G$ is $V(G)=V(H)\cup V(H')\cup V(H'')$. The edge set of $G$ is $E(G)=E(H)\cup E(H')\cup E(H'')\cup \{(l,p'),(l',p''),(l'',p):l \in B, p\in A\ \text{ and } (p,l)\in E(H)\} \cup \{(p,p'), (p',p''),(p'',p):p \in A\}$. In the graph $G$, all degrees are bounded by $(1+o(1))\sqrt{k}$. Moreover, given two vertices $p,q$ from $A\cup A'\cup A''$, either they are joined by a $P_2$ (in case both vertices come from the same set), by a $P_1$ (if they are in different sets and are not twin vertices) or by an edge (if they are twin vertices). This shows that a valid labeling of $G$ must use at least $3k$ colors.

\medskip

\end{proof}

\section{Open problems} 

There is a big list of problems to investigate about the labelings defined in the present note. Virtually every question studied for the undirected or the classical directed $L(2,1)$-labelings can be asked in the newly introduced settings. This list includes determining the exact value of the parameters for specific classes of graphs and finding relations between $\lambda_S(G)$ and other graph parameters, as it was done, for instance, with the path covering number \cite{lu2013path}. Moreover, it would be interesting to determine the correct asymptotic values in the cases $S=\{P_1\}$ and $S=\{P_2,P_3\}$. In particular, we conjecture that the construction in Theorem \ref{const} can be improved to match the upper bound asymptotically:

\begin{conj}
There is an oriented graph $G$ for which each indegree and outdegree is bounded by $(1+o(1))k$ and $\overrightarrow{\lambda}_{2,1}(G)\geq 2k^2 + O(k)$.
\end{conj}

\medskip

\section{Acknowledgments}

The research that led to this paper started in WoPOCA 2018, which was financed by FAPESP - S\~ao Paulo Research Foundation (2013/03447-6) and CNPq - The Brazilian National Council for Scientific and Technological Development (456792/2014-7). The research of the second author was partially supported by NKFIH grants \#116769, \#117879 and \#126853.

\bibliographystyle{plain}
\bibliography{refs}{}

\end{document}